\documentclass[reqno,10pt,notitlepage]{amsart}
\usepackage{amsmath,amsfonts,amsthm,amssymb}
\usepackage{enumerate}
\usepackage{indentfirst}
\usepackage[dvips]{graphicx}
\usepackage[all]{xy}
\usepackage{lscape}
\usepackage{a4wide}

\input{xypic.tex} \xyoption{all}

\newcommand{\Fcal}{\mathcal{F}}

\newcommand{\Sym}{\mathrm{Sym}}
\newcommand{\OO}{\mathcal{O}}
\newcommand{\WW}{\mathcal W}

\renewcommand{\P}{\mathbb{P}}

\DeclareMathOperator{\sing}{Sing}

\newtheorem{theorem}{Theorem}[section]

\newtheorem{corollary}[theorem]{Corollary}
\newtheorem{proposition}[theorem]{Proposition}

\theoremstyle{definition}

\newtheorem{remark}[theorem]{Remark}
\newtheorem{example}[theorem]{Example}

\theoremstyle{plain}

\font\smallrm=cmr8

\begin{document}

\author{MAUR\'ICIO CORR\^{E}A J.r}
\address{Universidade Federal de Vi\c cosa  \\
Departamento de Matem\'atica\\
Avenida P.H. Rolfs, s/n\\
36570--000 Vi\c cosa MG Brasil} \email{mauricio.correa@ufv.br}

\author{THIAGO FASSARELLA}
\address{Universidade Federal Fluminense  \\
Instituto de Matem\'{a}tica, Departamento de An\'{a}lise. \\
Rua M\'{a}rio Santos Braga, s/n\\
Valonguinho \\
24020--140  Niter\'{o}i RJ Brasil} \email{tfassarella@id.uff.br}



\title[{\smallrm{}On the order of the automorphism group of  foliations}]
{On the order of the automorphism group of  foliations}

\begin{abstract}
Let  $\Fcal$ be a holomorphic foliation with ample canonical bundle on a smooth projective surface $X$. We obtain
 an upper bound on the order of its automorphism group  which depends only on $K_{\Fcal}^{2}$ and $K_{\Fcal}K_{X}$, provided this group is
  finite. Here, $K_{\Fcal}$ and $K_{X}$ are the canonical bundles of $\Fcal$ and $X$, respectively.



\end{abstract}
\maketitle

\section{Introduction}

A finiteness theorem for the automorphism group of a smooth curve of genus at least $2$ goes back to the $19$th century with H. A. Schwarz, F. Klein and H. Poincar\'e.  The order of this group was studied by A. Hurwitz who obtained  in \cite{H} the following well--known bound
$$
|\mathrm{Aut}(C)|\leq 42\deg(K_C),
$$
where $K_{C}$ is the canonical bundle of the curve $C$.  In  arbitrary dimension, the finiteness of the automorphism group, or more generally, of the group of birational self-maps was proved by A. Andreotti in \cite{A} for varieties of general type,  and a bound in the above sense, was obtained in the $2$-dimensional case. He  showed that there exists a universal function $f$ so that
$$
|\mathrm{Aut}(X)|\leq f(c_{2}),
$$
where $c_{2}$ is the second Chern class of the tangent bundle of the smooth surface $X$.  The function $f$ given by Andreotti is of exponential type, and since then this estimate has been improved  in \cite{HuSa, Cor, X1, X2} and extended for higher dimensions by many others in \cite{HS,CaS,Ez,X3,Ts,Za,HMX}, placed in chronological order.



In the present paper, we are interested in the subgroup of the automorphism group, or even of the group of birational self--maps,
that preserves a holomorphic foliation $\Fcal$ on a smooth projective surface $X$.  One of the first progress in this direction has been made by J. V. Pereira and P. F. Sanchez in \cite{PS}. They  showed that the
 group ${\rm Bir}(\Fcal)$  of birational self-maps preserving  a foliation of general type on a
projective surface is finite. Therefore, we can  naturally to raise
the question if there exists an upper bound for the order of ${\rm
Bir}(\Fcal)$  in the sense of the previous one.

One purpose of this work is to answer this question for a
particular class  of general type foliations, namely, those whose
 the canonical bundle $K_{\mathcal G}$ of a reduced model $\mathcal G$ of $\Fcal$ is ample, see Corollary \ref{2dimensional}. This follows as consequence
 of Theorem \ref{Teorema principal}, our main result,  where we obtain an explicit upper bound on the order of  ${\rm Aut}(\Fcal)$ in which depends
only on intersections of the canonical bundles of $\mathcal F$ and $X$. More precisely, this bound is a function of  $K_{\Fcal}^{2}$ and $K_{\Fcal}K_{X}$.  In our theorem, a finiteness hypothesis on  ${\rm Aut}(\Fcal)$ is necessary.  We point out that foliations on projective surfaces having infinite automorphism group are very rare and they have been classified, up to birational maps,  in \cite{CF} by S. Cantat and C. Favre.  As a consequence of this classification, in Proposition \ref{classification} we classify the foliations on projective surfaces having infinite automorphism group and ample canonical bundle.

This paper is organized as follows.  In Section \ref{Foliation}, we recall some basic concepts on
 codimension one foliations as well their canonical bundles. In Section \ref{Web section}, we obtain a polynomial upper bound on the order of the automorphism group of a codimension one $k$-web of degree $d$ on $\P^{N}$ in terms of $k,d$ and $N$, and introduce their dual pairs. Finally, in
Section  \ref{proofs},  we prove our main results.

\bigskip
\noindent\textsc{Acknowledgements}.
We are greatly indebted to Charles Favre who have shown us the way to prove Proposition \ref{classification}. We also wish to thank Javier Rib\'on, Maycol Falla and Nivaldo Medeiros for many stimulating conversations on this subject.

\section{Codimension One Foliations}\label{Foliation}

\subsection{Basic Definitions}
Let $X$ be a smooth complex variety of dimension $n>1$.  A \textbf{codimension one foliation $\Fcal$ on $X$},  is
given by an open covering $\mathcal{U}=\{U_{i}\}$ of $X$ and
 $1$-forms $\omega_{i} \in
\Omega^{1}_{X}(U_{i})$ subject to the conditions:
\begin{enumerate}
 \item For each non-empty intersection $U_{i} \cap U_{j}$ there exists a function $g_{ij} \in \OO^{*}_{U_{i} \cap U_{j}}$ such that
  $\omega_{i}= g_{ij} \omega_{j}$;
 \item For every $i$ the zero set of $\omega_{i}$ has codimension at least two;
 \item For every $i$, $\omega_{i}\wedge d\omega_{i}=0$.
\end{enumerate}

The $1$-forms $\{\omega_i\}$ patch together to form a global section
$$
\omega=\{\omega_i\}\in{\rm H}^0(X,\Omega^{1}_{X}\otimes \mathcal L),
$$
where $\mathcal L$ is the line bundle over $X$ determined by the
cocycle $\{g_{ij}\}$. The \textbf{singular set} of $\Fcal$, denoted
by ${\rm{Sing}}(\Fcal)$, is the zero set of the twisted $1$-form
$\omega$.

Let $TX$ be the tangent sheaf of $X$. The \textbf{tangent sheaf} of
 $\Fcal$, induced by a twisted $1$-form $\omega\in
{\rm H}^0(X,\Omega^{1}_{X}\otimes \mathcal L),$ is defined on each open set $U\subset X$ by
$$
T_{\Fcal}(U)=\{v\in TX(U);\ i_{v}\omega=0\}.
$$
The dual of $T_{\mathcal F}$ is the cotangent sheaf of $\mathcal F$
and will be denoted by $T^*_{\mathcal F}$. The determinant of
$T^*_{\mathcal F}$, i.e. $(\wedge^{n-1} T^*_{\mathcal F})^{**}$, is called the \textbf{canonical bundle} of $\mathcal F$ and
is denoted by $K_{\mathcal F}$. If $n=2$, we can use vector fields instead $1$-forms. The foliation can be given by  $v_{i}\in TX(U_{i})$ with codimension two zero set and satisfying $v_{i}=f_{ij}v_{j}$, where $f_{ij}\in \mathcal O^{*}_{U_{i}\cap U_{j}}$. In this case, $K_{\Fcal}$ is the line bundle determined by the cocycle $\{f_{ij}\}$. For example, if $\Fcal$ is given by a (non-trivial) global vector field $v\in{\rm H}^{0}(X,TX)$ and $D$ denotes the zero divisor of $v$ (possibly empty if $v$ has codimension two zero set) then $K_{\Fcal}=\mathcal O_{X}(-D)$.  




The condition $(3)$ together with Frobenius Theorem imply that for every point  in the complement of ${\rm Sing} (\Fcal)$,  there exists a unique germ of smooth codimension one hypersurface $V$  invariant by $\Fcal$, i.e., satisfying $i^{*}(\omega)=0$ where $i:V \longrightarrow X$ is the inclusion. Analytic continuation of these subvarieties describes the {\bf leaves} of $\Fcal$. As usual, we will make abuse of notation by writing $(T_{\Fcal})_{x}$ for the tangent space to the leaf passing through $x$ or for the stalk of the sheaf $T_{\Fcal}$ at $x$.

We say that $\varphi:X\longrightarrow X$ preserves $\Fcal$ if it sends leaves on leaves, that is if $\varphi^{*}\omega$ defines the same foliation $\Fcal$. Through this work we want to analyze the order of the  groups ${\rm Aut}(\Fcal)$ and ${\rm Bir}(\Fcal)$, which are respectively defined as the maximal subgroup of automorphisms $\mathrm{Aut}(X)$ and of birational self--maps $\mathrm{Bir}(X)$ that preserve $\Fcal$. It has been proved by  J. V. Pereira and P. Sanchez in \cite{PS}  the following finiteness theorem for general type foliations on surfaces.

\begin{theorem}\label{Pereira-Sanchez}
If $\Fcal$ is a foliation of general type on a smooth projective surface then $\mathrm{Bir}(\Fcal)$ is finite.
\end{theorem}

A foliation on a surface $X$ is said of \textbf{general type} if the Kodaira dimension of $\Fcal$,  ${\rm Kod }(\mathcal F)$, is equal to $2$. The concept of Kodaira dimension for holomorphic foliations has been introduced independently by L. G. Mendes and M. McQuillan. For more information on the subject see \cite{Bru}. For convenience to the reader we will recall it in the next few lines.

An isolated singularity $x$ of $\mathcal F$ is called {\bf reduced} if the eigenvalues of the linear part $Dv(x)$ are not both zero and their quotient, when defined, is not a positive rational number. A foliation is called reduced if all the singularities are reduced.  A remarkable theorem by A. Seidenberg (see \cite{Sei}) says that there exists a sequence of blowing--ups $\pi: \tilde{X} \longrightarrow X$ over singularities of $\Fcal$ such that the induced foliation $\pi^{*}(\Fcal)$ in $\tilde{X}$ has only reduced singularities. Any reduced foliation birationaly equivalent to $\Fcal$ is called {\bf reduced model} of $\Fcal$.

We could define the Kodaira dimension of $\Fcal$ as the Kodaira dimension of the line bundle $K_{\Fcal}$, for short ${\rm Kod}(K_{\Fcal})$. But this is not a birational invariant for $\Fcal$. For example, the radial foliation (lines passing through a point) on $\P^{2}$ has canonical bundle $K_{\Fcal}=\mathcal O_{\P^{2}}(-1)$  but a suitable birational transformation $\P^{2} \dashrightarrow \P^{2}$ defined by polynomials of high degree transforms it into a foliation $\mathcal G$  with canonical bundle $K_{\mathcal G}=\mathcal O_{\P^{2}}(d-1)$.  Fortunately, ${\rm Kod}(K_{\Fcal})$ is a birational invariant for foliations having only reduced singularities. So if $\Fcal$ is a  foliation on the complex surface $X$, and $\mathcal{G}$ is any reduced model of $\Fcal$, the Kodaira dimension of $\Fcal$ is defined  as ${\rm Kod}(K_{\mathcal G})$.


\subsection{Foliations with $K_{\Fcal}$ ample on surfaces}  One of the main propose of this paper is to give a universal bound for the order of ${\rm Aut}(\Fcal)$ by supposing this group is finite and also  that $K_{\Fcal}$ is ample. Foliations having infinite automorphism group were classified up to birational maps by S. Cantat and C. Favre in \cite{CF} in the 2-dimensional case.  In what follows we will make use of their ideas to classify the foliations on projective surfaces having infinite automorphism group and ample canonical bundle.

Let $X$ be a projective surface and $f:X \longrightarrow X$ an automorphism. Fix an ample line bundle $\mathcal L$ on $X$. The degree of $f$ with respect to $\mathcal L$ is defined as
\[
{\rm deg}_{\mathcal L}(f)= f^{*}\mathcal L \cdot \mathcal L.
\]
It is known that the asymptotic behavior of the sequence  ${\rm deg}_{\mathcal L}(f^{n})$ is independent of the chosen line bundle $\mathcal L$. See for example the proof of Proposition 3.1 in \cite{BFJ}. Furthermore, this sequence either can be bounded or can have growth of the following types:   linear, quadratic  or  exponential. As we will see below, the case in which this sequence is bounded is particularly interesting for us. Under this hypothesis the automorphism $f$ is the time-$1$ map associated to a flow of a global vector field. For detail see \cite[p. 7]{CF} and references therein.

Now let $\Fcal$ be a foliation on $X$ with $K_{\Fcal}$ ample and infinite group  ${\rm Aut}(\Fcal)$. By \cite[Theorem 1.1]{CF} there exists an element $f\in  {\rm Aut}(\Fcal)$ with infinite order. Since $K_{\Fcal}$ is ample and $f^{*}(K_{\Fcal})=K_{\Fcal}$ then the sequence of degrees ${\rm deg}_{ K_{\Fcal}}(f^{n})$ is bounded. Hence $f$ is the time-$1$ flow of a nontrivial global vector field $v$. 

This implies that $\Fcal$ is preserved by a flow of a global vector field $v\in{\rm H}^{0}(X,TX)$. Indeed, let ${\rm Aut}_{0}(X)$ be the connected component of the identity of the complex Lie group ${\rm Aut}(X)$ and  ${\rm Aut}_{0}(\Fcal)= {\rm Aut}_{0}(X)\cap {\rm Aut}(\Fcal)$ be the closed Lie subgroup of  ${\rm Aut}_{0}(X)$.  If $\varphi_{t}$ is the time-$t$ flow of $v$, then $\varphi_{n}=f^{n}\in {\rm Aut}_{0}(\Fcal)$ for all $n\in \mathbb N$. This is enough to ensure that  the Lie algebra of ${\rm Aut}_{0}(\Fcal)$ is nontrivial. 

Since $\Fcal$ is preserved by a flow, it follows from \cite[Proposition 3.8]{CF} that $\Fcal$ is one of the following types:
 turbulent, Riccati, rational fibration, elliptic fibration, linear foliation on a torus or  up to a birational map, $\Fcal$ is preserved by the flow of a vector field on  $\P^{1}\times \P^{1}$.  All these cases, unless the last one, cannot happen under the hypothesis of $K_{\Fcal}$ to be ample.  In fact, if $\Fcal$ is turbulent or Riccati and $F$ is a generic fiber of the elliptic or  rational  fibration then $K_{\Fcal}F=0$ (see \cite[p. 23]{Bru}).  If $\Fcal$ is a rational fibration then the restriction of $K_{\Fcal}$ to a generic leaf $L\cong \mathbb P^{1}$ is $\mathcal O_{\mathbb P^{1}}(-2)$. Finally,  if $\Fcal$ is a linear foliation on a torus then $K_{\Fcal}=\mathcal O_{X}$. 

Suppose $\Fcal$ is birational  to a foliation preserved by the flow of a vector field on  $\P^{1}\times \P^{1}$. The birational map from $X$ to $\P^{1}\times \P^{1}$ is obtained in the following way. When $X$ is rational we can contract $(-1)$-curves to arrive in a minimal surface, which must be either $\P^{2}$  or a Hirzebruch surface $\Sigma_{n}$, $n\ge 0$.  We note that, at this point, $K_{\Fcal}$ is still ample. Blowing-up a singular point of $v$ we can replace $\P^{2}$ by $\Sigma_{1}$.  To arrive in $\P^{1}\times \P^{1}$ we have to make flipping of fibers that contains singularities of $v$ (see \cite[p. 87]{Bru}).  The resulting foliation, still denoted by $\Fcal$, on $\P^{1}\times \P^{1}$ is invariant by the flow of a  vector field $v_{1}\oplus v_{2}$. In addition, the argument of S. Cantat and C. Favre in the end of the proof of Proposition 3.8 in \cite{CF} shows that $\Fcal$ is given by a global vector field if $v$ is not tangent to a foliation by rational curves. We summarize the above discussion in the following proposition.

\begin{proposition}\label{classification}
 If $\Fcal$ is a foliation on a smooth projective surface $X$ with $K_{\Fcal}$ ample and  ${\rm Aut}(\Fcal)$ infinite, then up to a birational map, $\Fcal$ is preserved by the flow of a  vector field $v=v_{1}\oplus v_{2}$ on $\P^{1}\times \P^{1}$. Moreover, if $v$ is not tangent to a foliation by rational curves then $\Fcal$ is given by a global vector field on $\P^{1}\times \P^{1}$.
\end{proposition}

\begin{example}
Let us give an example of a foliation on $\mathbb P^{2}$ with $K_{\Fcal}$ ample and ${\rm Aut}(\Fcal)$ infinite. Consider the foliation $\Fcal$ of degree two on $\P^{2}$ given in homogeneous coordinates by the $1$-form
 \[
 \omega=-byz^{2}dx+(bxz^{2}+azy^{2})dy-ay^{3}dz
 \]
where $a,b\in \mathbb C^{*}$. This foliation has ample canonical bundle $K_{\Fcal}=\mathcal O_{\P^{2}}(1)$ and is preserved by the flow of the global vector field $\displaystyle{v=y\frac{\partial}{\partial x}}$. In fact, the flow preserves the $1$-forma $\omega$ since $L_{v}\omega = 0$ where $L_{v}$ is the Lie derivative. Blowing--up two singular points of $\Fcal$, contained in $(y=0)$, and contracting a $(-1)$-curve we arrive in a foliation given by the $1$-form $bdy-adx$ in an affine chart of  $\P^{1}\times \P^{1}$ which is preserved by the flow of $\displaystyle{\frac{\partial}{\partial x}}$. We leave the details to the reader.
\end{example}

\section{Codimension one Webs}\label{Web section}

\subsection{Basic Definitions}

Let $X$ be a smooth complex variety of dimension $n>1$.  A \textbf{codimension one $k$-web $\WW$ on $X$} is given by an open
covering $\mathcal{U}=\{U_{i}\}$ of $X$ and $k$-symmetric $1$-forms
$\omega_{i} \in \Sym^{k}\Omega^{1}_{X}(U_{i})$ subject to the
conditions:
\begin{enumerate}
 \item For each non-empty intersection $U_{i} \cap U_{j}$ there exists a non-vanishing function $g_{ij} \in \OO_{U_{i} \cap U_{j}}$ such that
  $\omega_{i}= g_{ij} \omega_{j}$;
 \item For every $i$ the zero set of $\omega_{i}$ has codimension at least two;
 \item For every $i$ and a generic $x\in U_i$, the germ of $\omega_{i}$ at $x$ seen as homogeneous polynomial of degree
 $k$ in the ring $\mathcal O_x [dx_{1},...,dx_{n}]$ is square--free;
 \item For every $i$ and a generic $x \in U_{i}$, the germ of $\omega_{i}$ at $x$
 is a product of $k$ germs of  $1$-forms $(\omega_{i})_{x}=\beta_{1}\cdots  \beta_{k}$ satisfying the integrability condition $\beta_{i}\wedge d\beta_{i}=0$.
\end{enumerate}

The $k$-symmetric $1$-forms $\{\omega_i\}$ patch together to form a
global section
$$
\omega=\{\omega_i\}\in
{\rm H}^0(X,{\rm Sym}^{k}\Omega^{1}_{X}\otimes \mathcal L),
$$
where $\mathcal L$
is the line bundle over $X$ determined by the cocycle $\{g_{ij}\}$. Two global sections in ${\rm H}^0(X,{\rm Sym}^{k}\Omega^{1}_{X}\otimes \mathcal L)$ determine the same web if and only if they differ by the multiplication by an element of ${\rm H}^{0}(X,\mathcal O_{X}^{*})$.
The \textbf{singular set} of $\WW$, denoted by ${\rm{Sing}}(\WW)$,
is the zero set of  $\omega$.

We say that $x\in X$ is a \textbf{smooth point} of $\WW$,
$x\in \WW_{sm}$, if $x\notin {\rm{Sing}}(\WW)$ and the germ of
$\omega$ at $x$ satisfies the conditions described in $(3)$ and
$(4)$ above. For any smooth point $x$ of $\WW$ we have $k$ distinct
(not necessarily in general position) linearly embedded subspaces of
dimension $n-1$ passing through $x$. Each one of these subspaces
will be called $(n-1)$-plane tangent to $\WW$ at $x$ and denoted by
$T^1_x\WW,...,T^k_x\WW$. Furthermore, these conditions ensure that in a neighborhood of $x$ there are $k$ distinct codimension one foliations $\Fcal^{1},...,\Fcal^{k}$ satisfaing $T\Fcal_{x}^{j}=T^{j}_{x}\WW$.

In the case that $X=\P^{N}$ the \textbf{degree} of $\WW$,
 denoted by $\deg(\WW)$, is geometrically defined as the degree of the tangency locus between $\WW$ and a generic $\mathbb P^{1}$
  linearly embedded in $\P^{N}$. If $i:\mathbb P^{1} \hookrightarrow  \P^{N}$ is the inclusion, then ${\rm deg}(\WW)$ is the degree of the
   zero divisor of the twisted $k$-symmetric $1$-form
   $i^*\omega\in {\rm H}^0(\mathbb P^{1},{\rm Sym}^{k}\Omega^{1}_{\mathbb P^{1}}\otimes \mathcal L|_{\mathbb P^{1}})$.
   Since $\Omega^{1}_{\mathbb P^{1}}=\mathcal O_{\mathbb P^{1}}(-2)$ one obtains $\mathcal L=\mathcal O_{\mathbb P^{n}}(d+2k)$ where $d=\deg(\WW)$. From the Euler sequence\footnote{For properties of symmetric powers see \cite[Propositions A2.2 and  A2.7]{E}.} we can deduce the following
\[
0 \longrightarrow \mathrm{Sym}^{k-1} \left( \mathcal O_{\P^{N}}(1)^{\oplus (N+1)}\right) \otimes \mathcal O_{\P^{N}} \longrightarrow    \mathrm{Sym}^{k} \left( \mathcal O_{\P^{N}}(1)^{\oplus (N+1)}\right)  \longrightarrow  \mathrm{  Sym}^{k}T\P^{N}  \longrightarrow 0.
\]
Taking the dual sequence and tensorizing by $\mathcal O_{\P^{N}}(d+2k)$  we can see that a  $k$-web given by
$$\omega\in {\rm H}^{0}(\P^{N}, \mathrm{  Sym}^{k}\Omega_{\P^{N}}^{1}\otimes \mathcal O_{\P^{N}}(d+2k))$$
can be represented in homogeneous coordinates by a $k$-symmetric form, still denoted by $\omega$
\[
\displaystyle{\omega(x)=\sum_{|I|=k}A_{I}(x) dx_{0}^{i_{0}}\cdots
dx_{N}^{i_{N}}},
\]
whose the coefficients $A_{I}$ are homogeneous polynomials of degree $d+k$ in $x_{0},...,x_{N}$. Two $k$-symmetric forms in homogeneous coordinates define the same $k$-web if they differ by a constant.

We finish this section remarking that a $k$-web on a possible singular projective variety $X$ is a $k$-web on its smooth locus which extends to a global web on any of its desingularizations.  We also observe that  a $1$-web is  a codimension one foliation, as defined in  Section \ref{Foliation}.

\subsection{Automorphism Group of Projective Webs}

Let $\WW$ be a codimension one $k$-web on the complex projective
space $\P^{N}$ defined by the twisted $k$-symmetric $1$-form $\omega$. We say that $T\in \mathrm{Aut}(\P^{N})$ preserves $\WW$ if the pullback $T^{*}\omega$ determines the same $k$-web $\WW$. As remarked in the last section, $\omega$ and $T^{*}\omega$ must  differ by the multiplication by an element of ${\rm H}^{0}(\mathbb P^{N},\mathcal O_{\mathbb P^{N}}^{*})$, that is, by a nonvanishing constant. The automorphism group of $\WW$, denoted by
$\mathrm{Aut}(\WW)$, is the maximal subgroup of $\mathrm{Aut}(\P^{N})$
that preserves $\WW$.

\begin{proposition}\label{boundweb}
 Let $\WW$ be a $k$-web on $\P^{N}$, $N\ge 2$, of degree $d$ and finite  automorphism group. Then
 \[
 |{\rm{Aut}}(\WW)|\le (d+2k)^{(N+1)^{2}-1}.
 \]
\end{proposition}

\begin{proof}
We first consider the $k$-symmetric form defining $\WW$ in homogeneous coordinates
\[
\displaystyle{\omega(x)=\sum_{|I|=k}A_{I}(x) dx_{0}^{i_{0}}\cdots
dx_{N}^{i_{N}}},
\]
where $A_{I}(x_{0},...,x_{N})$ is a homogeneous polynomial of degree
$d+k$ for all $I=(i_{0},...,i_{N})$.


Writing an element of ${\rm{Aut}}(\WW)$ as a $(N+1)\times (N+1)$
matrix we may think of its coordinates as homogeneous coordinates of
a point $T=(...:a_{ij}:...)\in\P^{r}$ where $r=(N+1)^{2}-1$.

Since $\omega$ and $T^{*}\omega$ differ by the multiplication by a nonvanishing constant, for fixed $x=(x_{0},...,x_{N})$ the entries of $T$ must satisfy the equation $\omega(x)\wedge (T^{*}\omega)(x)=0$.
Now we can rewrite the pullback of $\omega$ by $T$,
\[
\displaystyle{T^{*}\omega(x)=\sum_{|I|=k}A_{I}(...,\sum a_{ij} x_{j},...) (\sum a_{0j} dx_{j})^{i_{0}}\cdots (\sum a_{Nj} dx_{j})^{i_{N}}}
\]
in terms of the basis $\{dx_{0}^{i_{0}}\cdots dx_{N}^{i_{N}}\;;\; |I|=k\}$ to obtain
\[
\displaystyle{T^{*}\omega(x)=\sum_{|I|=k}B_{I}^{x}(...,a_{ij},...)
dx_{0}^{i_{0}}\cdots dx_{N}^{i_{N}}},
\]
where $B_{I}^{x}(...,a_{ij},...)$ is a homogeneous polynomial of
degree $d+2k$ in $T=(...,a_{ij},...)$ for each $x\in\mathbb C^{n+1}$
and for each index  $I$. So the equation  $\omega(x)\wedge
(T^{*}\omega)(x)=0$ is equivalent to the vanishing of the
polynomials
$H_{IJ}^{x}(...,a_{ij},...)=A_{I}(x)B_{J}^{x}(...,a_{ij},...)-A_{J}(x)B_{I}^{x}(...,a_{ij},...)$
for all $I,J$ with $|I|=|J|=k$.

Therefore ${\rm{Aut}}(\WW)$ is an intersection of  hypersurfaces
$Z_{IJ}^{x}=(H_{IJ}^{x}=0)$ of degree $d+2k$ in $\P^{r}$. By
hypothesis,   ${\rm{Aut}}(\WW)$ is finite. Then we can choose $r$ among the
 $Z_{IJ}^{x}$, say $Z_{1},...,Z_{r}$ such that $Z_{i}$ does
not contain any irreducible component of $Z_{1}\cap \cdots \cap
Z_{i-1}$ for each $i=1,...,r$, and $\displaystyle{{\rm{Aut}}(\WW)= \cap_{i=1}^{r}Z_{i}}$.
It follows from B\'ezout's Theorem that
\[
|{\rm{Aut}}(\WW)|\le (d+2k)^{(N+1)^{2}-1}.
\]

\end{proof}

In particular, we obtain a polynomial upper bound for the order of
the self-birational group of a generic  foliation $\Fcal$ on
$\P^2$ depending only on its degree.

\begin{corollary}\label{boundweb cor}
 If $\Fcal$ is a reduced foliation on $\P^{2}$ of degree $d>1$, then ${\rm Bir}(\Fcal)={\rm Aut}(\Fcal)$ and
 \[
 |{\rm{Bir}}(\Fcal)|\le (d+2)^{8}.
 \]
\end{corollary}
\begin{proof}

Since  ${K}_{\Fcal}=\mathcal{O}_{\P^{2}}(d-1)$ is ample, follows from \cite[Theorem 1 p.75]{Bru} that there exists a unique minimal model for  $(\P^{2},\Fcal)$. The minimal model is obtained in two steps: first resolve the singularities of the foliation and then contract exceptional invariant curves.  But the singularities are already reduced and $\P^{2}$ does not admit $(-1)$-curves. Therefore $(\P^{2},\Fcal)$ is minimal and consequently ${\rm{Bir}}(\Fcal)={\rm{Aut}}(\Fcal)$. To conclude the proof we have just to apply Theorem \ref{Pereira-Sanchez} and Proposition \ref{boundweb}.
\end{proof}

\subsection{Duality}

Let $\check{\mathbb P}^{N}$ denote the projective space parametrizing hyperplanes in $\mathbb P^{N}$ and $\mathcal I$ be the incidence variety, that is 
$$
\mathcal I=\{(x,H)\in \mathbb P^{N}\times \check{\mathbb P}^{N} \;;\; x\in H\}. 
$$
If $X\subset \P^{N}$ is a projective variety  of dimension $n$, the \textbf{conormal variety} ${\rm Con(X)}$ is the closure in $\mathcal I$ of the set of pairs $(x,H)$ such that $x$ is a smooth point of $X$ and $H$ is a hyperplane containing the tangent plane $T_{x}X$.

Now let $\WW$ be an
irreducible codimension one $k$-web on $X$. The conormal
variety of the pair $(X,\WW)$ is the subvariety ${\rm{Con}}(X,\WW)$ of
$\mathcal I$ defined by
\[
{\rm{Con}}(X,\WW)=\overline{\{(x,H)\in \mathcal I\;;\; x \in \WW_{sm}
\hspace{0.1cm} {\rm{and}} \hspace{0.1cm} \exists \hspace{0.1cm}i,
\hspace{0.1cm} 1 \leq i \leq k ,\hspace{0.1cm} H\supset T^i_x\WW
\hspace{0.1cm}  \}},
\]
where the overline in the right side means the Zariski closure in
$\mathcal I$. We note that for $x\in \WW_{sm}$ the fiber
$\pi^{-1}(x)\cap {\rm{Con}}(X,\WW)$ is a union of $k$ linear spaces
of dimension $N-n$, then ${\rm{Con}}(X,\WW)$ has codimension $N-1$ in
$\mathcal I$.

Using the natural identification of $\mathcal I$ with $\P
(T^{*}\P^{N})$ the conormal ${\rm{Con}}(X,\WW)$ can be also
characterizad by the following conditions (see \cite[Section 1.4]{PiPer}):
\begin{enumerate}
\item ${\rm{Con}}(X,\WW)$ is irreducible;
\item $\pi ({\rm{Con}}(X,\WW)) = X$;
\item For generic $x \in X$ the fiber $\pi^{-1}(x)\cap {\rm{Con}}(X,\WW)$ is a union
of linear subspaces corresponding to the projectivization of the conormal bundles in $\P^{N}$
of the leaves of $\WW$ passing through $x$.
\end{enumerate}

Given $x\in \WW_{sm}$ we have the superposition of $k$ distinct
codimension one foliations in a neighborhood $U$ of $x$ in $X$. The
lift of these foliations to $\mathcal I$ determines a codimension
one foliation $\mathcal F_{U}$ on $\pi^{-1}(U)\cap {\rm{Con}}(X,\WW)$
in which is just the foliation obtained by the restriction of the
contact form $\alpha \in {\rm{H}}^{0}(\mathcal I, \Omega_{\mathcal
I}^{1} \otimes \mathcal O_{\mathcal I}(1))$ to $\pi^{-1}(U)\cap
{\rm{Con}}(X,\WW)$ (see \cite[p. 39]{PiPer} for details). So we get a
codimension one foliation $\Fcal$ globally defined on
${\rm{Con}}(X,\WW)$ such that the leaves are the lifting of leaves of
$\WW$.

In order to obtain the dual pair $\mathcal D(X,\WW)$ we consider the
variety $Y=\check{\pi}({\rm{Con}}(X,\WW))\subset \check{\P}^{N}$. Let
$s$ be the number of irreducible components of the fiber
$\check{\pi}^{-1}(H)\cap {\rm{Con}}(X,\WW)$ for generic
$H\in Y$. If $\dim Y>1$, over regular values of
$\check{\pi}$ the direct image of $\Fcal$ can be identified with the
superposition of $s$ distinct codimension one foliations. Therefore
we arrive in a codimension one web $\WW^{\lor}$ on $Y$. If $\dim
Y=1$ the leaves of $\Fcal$ are projected in points. In this case we
adopt the convention that on the irreducible curve $Y$ there is only
one web $\WW^{\lor}$, the $1$-web  which has as leaves the points of
$Y$. We shall say that $\mathcal D(X,\WW)=(Y,\WW^{\lor})$ is the
{\bf{dual pair}} of $(X,\WW)$.

The {\bf{characteristic numbers}} associated to the pair $(X,\WW)$
are
\[
d_{i}(X,\WW)= \int_{\mathcal
I}[{\rm{Con}}(X,\WW)][{\rm{Con}}(\P^{i})],
\]
where $i=0,...,N-1$ and $\P^{i}\subset \P^{N}$. Notice that if $X=\mathbb P^{N}$ then $d_{0}(X,\WW)=k$ and $d_{1}(X,\WW)={\rm deg}(\WW)$. The last equality follows from the fact that ${\rm Con}(\mathbb P^{1})$ is the set of points $(x,H)$ such that $x\in \mathbb P^{1}$ and $H$ contains $\mathbb P^{1}$. So if $\mathbb P^{1}$ represents a generic line,  $x\in{\rm{Con}}(X,\WW)\cap{\rm{Con}}(\P^{1})$ if and only if  $x$ is contained in the tangency locus between $\WW$ and $\mathbb P^{1}$. We also observe that since
${\rm{Con}}(X,\WW)={\rm{Con}}\mathcal D (X,\WW)$ and ${\rm{Con}}(\P^{i})
= {\rm{Con}}(L)$ where $L\simeq \P^{N-i-1} \subset \check{\P}^{N}$
one obtains
\begin{eqnarray}\label{duality}
d_{i}(X,\WW)=d_{N-i-1}\mathcal D (X,\WW).
\end{eqnarray}


\section{Main Results}\label{proofs}

Given a smooth projective variety $X$ of dimension $n$ and a
codimension one foliation $\Fcal$ on $X$, with ample canonical
bundle $K_{\Fcal}$, the $m$-th pluricanonical map
$$
\phi_{m}: X \longrightarrow \P^{N}=\P{\rm{H}}^{0}(X,
K_{\Fcal}^{\otimes m})^{\lor}
$$
is an embedding for sufficiently large $m$. The map $\phi_{m}$ send
a point $p\in X$ to the hyperplane in $\P{\rm{H}}^{0}(X,
K_{\Fcal}^{\otimes m})$ consisting of the sections vanishing at $p$.
Let us denote by $X_{m}=\phi_{m}(X)\subset \P^{N}$  the embedded
variety and $\Fcal_{m}$ the respective foliation on $X_{m}$. Since
${\rm{Aut}}(\Fcal)$ acts naturally in $\P{\rm{H}}^{0}(X,
K_{\Fcal}^{\otimes m})$ and so in $\P{\rm{H}}^{0}(X,
K_{\Fcal}^{\otimes m})^{\lor}$, this action induces a monomorphism
of groups
\[
{\rm{Aut}}(\Fcal) \longrightarrow {\rm{PGL}}(N,\mathbb C)
\]
which image in ${\rm{PGL}}(N,\mathbb C)$ is exactly ${\rm{Lin}}(\Fcal_{m})$, the automorphisms of $\P^{N}$ leaving $X_{m}$ and $\Fcal_{m}$ invariant.

The following remark will be useful in the sequel. 

\begin{remark}\label{Mats}
 Matsusaka's Big Theorem states that given an ample line bundle $\mathcal L$ on a smooth projective variety $X$ of dimension $n$, there exists a positive integer $m_{0}$, depending only on the coefficients of the Hilbert polynomial of $\mathcal L$ such that $\mathcal L^{\otimes m}$ is very ample for $m\ge m_{0}$. This theorem was improved by J. Koll\'ar and T. Matsusaka in \cite{KM}  showing that $m_{0}$  depends only on the coefficients $t^{n}$ and $t^{n-1}$ in the Hilbert Polynomial $P(t)$. This implies that $m_{0}$ depends only on $\mathcal L^{n}$ and $\mathcal L^{n-1}K_{X}$. We will need of an effective version of Matsusaka's Big Theorem obtained by G. Fernandez del Busto in  \cite{Fe}. He proves that if  $\mathcal L$ is
an ample line bundle on a smooth projective  surface $X$ and
\begin{equation}\label{cota-efetiva}
m>k_{0}=\frac 12 \Bigl[\frac {(\mathcal L\cdot (K_X+4 \mathcal L)+1)^2}{\mathcal L^2}+ 3\Bigr]
\end{equation} then $\mathcal L^{\otimes m}$ is very ample.
\end{remark}

Now we are able to prove our main result. The idea of the proof is in some sense close to the argument of \cite[Theorem 1]{HS}. That is, first apply Matsusaka's Big Theorem to embed $X$, goes to the dual space and then use homogeneous coordinates to bound the order of ${\rm{Aut}}(\Fcal)$.  In our case, the dual of the pair $(X,\Fcal)$ is a codimension one web. 

\begin{theorem}\label{Teorema principal}
Let $\Fcal$ be  a  foliation on a smooth irreducible projective surface $X$. If $K_{\Fcal}$ is ample and ${\rm{Aut}}(\Fcal)$ is finite  then
 \[
|{\rm{Aut}}(\Fcal)|\le \left((3m^2+2m)K_{\Fcal}^{2}\right)^{(m^2K_{\Fcal}^2+2)^2-1}
 \]
 where $m= (K_{\Fcal}\cdot (K_{X}+4 K_{\Fcal})+1)^2+3K_{\Fcal}^{2}$.
\end{theorem}

\begin{proof}
Let us fix $m= (K_{\Fcal}\cdot (K_{X}+4 K_{\Fcal})+1)^2+3K_{\Fcal}^{2}$. By Remark \ref{Mats}, $K_{\Fcal}^{\otimes m}$ is very ample and hence the $m$-pluricanonical map is an embedding.    Let $X_{m}\subset \P^{N}$ be the embedded variety and $\Fcal_{m}$ the respective foliation on $X_{m}$. We will first prove that the dual web $\mathcal W_{m}=\Fcal_{m}^{\lor}$ is a codimension one web on $\check{\P}^{N}$. Since the dual variety of a curve, in which is not a line, is always a hypersurface,  we must prove that a generic leaf of $\Fcal_{m}$ is not a line.  If a leaf $L$ of $\Fcal_{m}$ is a line then $L$ must pass through a singular point of the foliation. In fact, if $L$ does not contain a singular point then the restriction of $K_{\Fcal_{m}}$ to this leaf coincides with the canonical bundle $K_{L}=\mathcal O_{\mathbb P^{1}}(-2)$ of $L$.  This contradicts the hypothesis of  $K_{\Fcal_{m}}$ be ample. Therefore if a generic leaf of $\Fcal_{m}$ is a line, there are an infinite quantity passing through the same singular point $x\in \sing(\Fcal_{m})$.  Since $X$ is smooth irreducible this is enough to ensure that $X$ is a linearly embedded $\mathbb P^{2}$ and $\Fcal_{m}$ is the foliation by lines passing through $x$. The last assertion follows from the fact that the tangence between two foliations is an algebraic subset of $X$.   But, in this case, $K_{\Fcal_{m}}=\mathcal O_{\mathbb P^{2}}(-1)$ is not ample.

The image of ${\rm{Aut}}(\Fcal)$ in
${\rm{PGL}}(N,\mathbb C)$ by the monomorphism
\[
{\rm{Aut}}(\Fcal) \longrightarrow {\rm{PGL}}(N,\mathbb C)
\]
is ${\rm{Lin}}(\Fcal_{m})$, the automorphisms of $\P^{N}$ leaving $X_{m}$ and $\Fcal_{m}$ invariant. Since
${\rm{Lin}}(\Fcal_{m})\simeq {\rm{Aut}}(\WW_{m})$ we have just to bound ${\rm{Aut}}(\WW_{m})$. By hypothesis
${\rm{Aut}}(\Fcal)$ is finite, and thus  one obtains the same conclusion for ${\rm{Aut}}(\WW_{m})$. It follows from Proposition \ref{boundweb} that
\[
|{\rm{Aut}}(\WW_{m})|\le (d+2k)^{(N+1)^{2}-1}.
\]

But from $(\ref{duality})$, 
\[
d=d_{1}(\check{\P}^{N},\WW_{m})=d_{N-2}(X_{m},\Fcal_{m})
\]
and
\[
k=d_{0}(\check{\P}^{N},\WW_{m})=d_{N-1}(X_{m},\Fcal_{m}).
\] 
Now we can write this numbers intrinsically.  The second one is the number of tangencies between $\Fcal_{m}$ and a generic hyperplane section of $X_{m}$. That is, 
\[
d_{N-1}(X_{m},\Fcal_{m})={\rm tang}(\Fcal_{m},\mathcal O_{X_{m}}(1))={\rm tang}(\Fcal,K_{\Fcal}^{\otimes m}). 
\]
Thus by \cite[Proposition 2 p.23]{Bru}, $d_{N-1}(X_{m},\Fcal_{m})=(m^{2}+m)K_{\Fcal}^{2}$. It is not hard to see that the first number $d_{N-2}(X_{m},\Fcal_{m})$ coincides with  ${\rm deg}(X_{m})$, hence $d_{N-2}(X_{m},\Fcal_{m})=m^{2}K_{\Fcal}^{2}$. Therefore 
\[
|{\rm{Aut}}(\WW_{m})|\le ((3m^{2}+2m)K_{\Fcal}^{2})^{(N+1)^{2}-1}.
\]

Since $N=h^{0}(X,K_{\Fcal}^{\otimes m})-1$, to conclude the proof of the theorem we need only to apply the following bound  given in  \cite[Theorem 3]{MM} 
\begin{equation}\label{cota-secoes}
h^0(X,K_{\Fcal}^{\otimes m})\leq m^2K_{\Fcal}^2+2.
\end{equation}

\end{proof}

\begin{corollary}\label{main corollary}
 Let $\Fcal$ be a foliation on a non--rational smooth irreducible projective surface $X$. If $K_{\Fcal}$ is ample then
 \[
|{\rm{Aut}}(\Fcal)|\le \left((3m^2+2m)K_{\Fcal}^{2}\right)^{(m^2K_{\Fcal}^2+2)^2-1}
 \]
 where $m= (K_{\Fcal}\cdot (K_{X}+4 K_{\Fcal})+1)^2+3K_{\Fcal}^{2}$.
\end{corollary}

\begin{proof}
The finiteness of ${\rm{Aut}}(\Fcal)$ is ensured by Proposition \ref{classification}. 
\end{proof}

In the following corollary, the pair $(X,\Fcal)$ represents a foliation $\Fcal$ on a smooth irreducible projective surface $X$.

\begin{corollary}\label{2dimensional}
There exists a function $f$, so that if $(X,\Fcal)$ is reduced and has ample canonical bundle $K_{\Fcal}$, then the minimal model $(Y,\mathcal G)$ has ample canonical bundle $K_{\mathcal G}$  and
\[
|{\rm{Bir}}(\Fcal)|\le f(K_{\mathcal G}^{2},K_{\mathcal G}K_{Y}).
\]
\end{corollary}

\begin{proof}
We proceed
 similarly to the proof of Corollary \ref{boundweb cor}. Let $(Y,\mathcal G)$ the minimal model of $(X,\Fcal)$.
  The minimal model is obtained contracting exceptional invariant curves $\pi:(X,\Fcal) \longrightarrow (Y,\mathcal G)$ and the contraction
  of an exceptional curve produces a reduced singularity or a regular point. Thus $\pi$ is a sequence of blowing--ups over reduced singularities or regular points.

We will show that amplitude of $K_{ \Fcal}$ implies amplitude of $K_{\mathcal G}$. In general if $\pi: (X,\Fcal ) \longrightarrow (Y_{1},{\mathcal G}_{1})$ is the blow up over a singularity $x$ of ${\mathcal G}_{1}$ and $l_{x}$ is the order of ${\mathcal G}_{1}$ on $E=\pi^{-1}(x)$ then $K_{\Fcal}=\pi^{*}(K_{{\mathcal G}_{1}})\otimes \mathcal O((1-l_{x})E)$, see \cite[p. $26$]{Bru}. If $x$ is a reduced singularity then $K_{\Fcal}=\pi^{*}(K_{{\mathcal G}_{1}})$ and if $x$ is a regular point then $K_{\Fcal}=\pi^{*}(K_{{\mathcal G}_{1}})\otimes \mathcal O(E)$. Thus  $K_{\Fcal}$  ample implies that  $K_{{\mathcal G}_{1}}$ is ample, as can be checked by using Nakai--Moishezon Criterion. This is enough to deduce the amplitude of $K_{\mathcal G}$.


To finish the proof of the corollary we observe that ${\rm{Bir}}(\Fcal)={\rm{Bir}}(\mathcal G)={\rm{Aut}}(\mathcal G)$ and apply Theorem \ref{Pereira-Sanchez} together with Theorem \ref{Teorema principal}.
\end{proof}

\begin{remark}
We remember that by definition,  $\Fcal$ is of general type if and only if  $\Fcal$ has a reduced model $\tilde{\Fcal}$ with big canonical bundle $K_{\tilde{\Fcal}}$. Corollary \ref{2dimensional} can be applied to foliations $\Fcal$ with the stronger hypothesis on $K_{\tilde{\Fcal}}$ to be ample.
\end{remark}

\end{document}